\documentclass[11pt]{amsart} 
\usepackage{amsmath,amsthm}
\usepackage{amssymb}
\usepackage{amscd}
\usepackage{graphicx}
\usepackage{color}

\newtheorem{theorem}{Theorem}[section]

\newtheorem{question}[theorem]{Question}

\theoremstyle{definition}

\newtheorem{example}[theorem]{Example}

\newtheorem{remark}[theorem]{Remark}

\newcommand{\tb}{\ensuremath{{\mbox{\tt tb}}}}

\begin{document}

\def \l {L(p, 1)}

 \title{On the Support Genus of Legendrian Knots}

 \author{Sinem  Onaran}

 \address{Department of Mathematics, Hacettepe University, 06800 Beytepe-Ankara, Turkey}
  \email{sonaran@hacettepe.edu.tr}
   \subjclass{57R17}

\begin{abstract}
In this paper, we show that any topological knot or link in $S^1 \times S^2$ sits on a planar page of an open book decomposition whose monodromy is a product of positive Dehn twists. As a consequence, any knot or link type in $S^1 \times S^2$ has a Legendrian representative having support genus zero. We also show this holds for some knots and links in the lens spaces $L(p,1)$.
\end{abstract}
\maketitle
 \setcounter{section}{0}

\section{Introduction} 

\noindent In 2002, Giroux showed that given any contact structure $\xi$ on a $3-$manifold $M$ there is an open book decomposition of the $3-$manifold such that the contact structure is transverse to the binding of the open book and it can be isotoped arbitrarily close to the pages. In this case we say contact $3-$manifold $(M, \xi)$ is supported by this open book decomposition. Further, Giroux showed that a contact $3-$manifold is Stein fillable (and hence tight) if and only if there is an open book decomposition for the contact $3-$manifold whose monodromy can be written as a product of positive Dehn twists, \cite{Gi2}. The purpose of this note is to construct the simplest open book decomposition of lens spaces whose monodromy is a product of positive Dehn twists and which contains a given knot or link on its page. The first result constructs explicitly such open book decomposition for  knots and links in $S^1 \times S^2$.   

\begin{theorem} Any topological knot or link in $S^1 \times S^2$ sits on a planar page of an open book decomposition whose monodromy is a product of positive Dehn twists.
 \label{thm1}
\end{theorem}

In \cite{On}, a new invariant of Legendrian knots in a contact $3-$manifold $(M, \xi)$ is defined using open book decompositions. The invariant is called the support genus $sg(L)$ of $L$ and it is defined as the minimal genus of a page of an open book decomposition of $M$ supporting $\xi$ such that $L$ sits on a page of the open book and the framings given by $\xi$ and the page agree. In \cite{On} it was shown that while any null-homologous knot with an overtwisted complement has $sg(L) = 0$, not all Legendrian knots have. Moreover, it was shown that any knot or link in the $3-$sphere $S^3$ sits on a planar page of an open book decomposition whose monodromy is a product of positive Dehn twists. This guarantees  a Legendrian representative of the given knot or link having support genus zero in the standard tight contact $S^3$. This paper addresses this theme for lens spaces.

According to \cite{El}, $S^1 \times S^2$ has a unique tight and Stein fillable contact structure. Then by \cite{Gi2}, the open book decomposition constructed in Theorem~\ref{thm1} supports this unique tight and Stein fillable contact structure. As a consequence of Theorem~\ref{thm1}, we have the following result on the support genus of Legendrian knots or links in $S^1 \times S^2$. 

\begin{theorem} Given a knot or link type $K$ in the tight contact $S^1 \times S^2$, there is a Legendrian representative $L$ of $K$ such that $sg(L) = 0$. 
  \label{thm12}
\end{theorem}

The proof of Theorem~\ref{thm1} does not extend to lens spaces (see Remark~\ref{remark2}); however, using another method we show the following.

\begin{theorem} Let $K$ be a knot or link type in $\l$ which is presented as a $2n$-pure braided plat as in Figure~\ref{knot2}, $n \in \mathbb{N}$. For $p > 2n - 2$, there is a tight contact structure $\xi$ on $\l$ and a Legendrian representative $L$ of $K$ in $(\l, \xi) $ such that $sg(L) = 0$.
  \label{thm2}
\end{theorem}

\begin{remark} Note that Theorem~\ref{thm1} generalizes to knots or links in the connect sums $\#_m S^1\times S^2$. This generalization is handy in studying contact geometry in dimension $5$. In \cite{AEB}, Acu, Etnyre and  Ozbagci use this result to study open book decompositions of $5-$manifolds and show that many sub-critically fillable contact $5-$manifolds are supported by iterated planar open book decompositions. Theorem~\ref{thm2} generalises to  $\#_m\l$ as well.
\end{remark}

In the following section we prove the main theorems and discuss examples. We conclude with a handful of questions related to the support genus of Legendrian knots in Section 3. The reader may refer to \cite{Ge}, \cite{Et2}  for background information on contact structures and open book decompositions. 

\textbf{Acknowledgement} I would like to thank John B. Etnyre for his interest in this work. I would like to thank Hansj\"org Geiges for helpful correspondence. This work was partially supported by the BAGEP award of the Science Academy, Turkey.

\section{Knots and links in $\l$} 

\noindent The lens space $\l$ can be obtained by a surgery along a single $(-p)$-framed unknot $U$ in $S^3$. Given any knot $K$ in $\l$, take a projection of $K$ in the surgery diagram for $\l$ such that the unknot $U$ intersects the plane of the projection transversely at a single point.  The knot $K$ in the surgery diagram of $\l$ can be given as in Figure~\ref{knot1} which is called a mixed link diagram of the link $K \cup U$.  When $p=0$, surgery along $0$-framed unknot $U$ gives $S^1 \times S^2$.
\begin{figure}[h!]
\begin{center}
 \includegraphics[width=4cm]{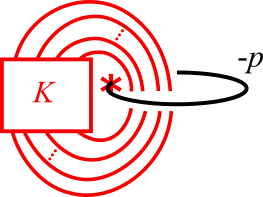}
 \caption{Knots in the surgery diagram of  $\l$.}
  \label{knot1}
\end{center}
\end{figure}

After isotoping the $(-p)$-framed unknot as in Figure~\ref{knot2}, we may present the knot $K$ as a pure braided plat by using Lemma $3.3$ of \cite{On}. Pure braided plat of $K$ is a shifted $2n$-plat given in Figure~\ref{knot2} where $2n$-braid is a pure braid.

\begin{figure}[h!]
\begin{center}
 \includegraphics[width=13cm]{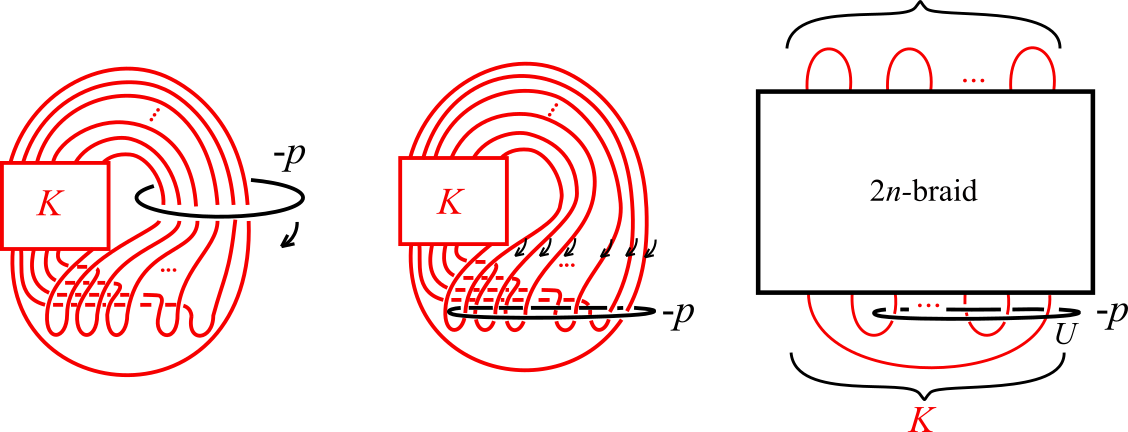}
 \caption{Pure braided plat presentation of $K$ in $\l$ where $2n$-braid is a pure braid.}
  \label{knot2}
\end{center}
\end{figure}

Given a knot $K$ in $\l$ as in Figure~\ref{knot2}, we decompose $K$ and the isotoped unknot $U$ in terms of the standard generators of the pure braid group. See Figure~\ref{knot3} for a decomposition of $U$. 

\begin{figure}[h!]
\begin{center}
 \includegraphics[width=11.5cm]{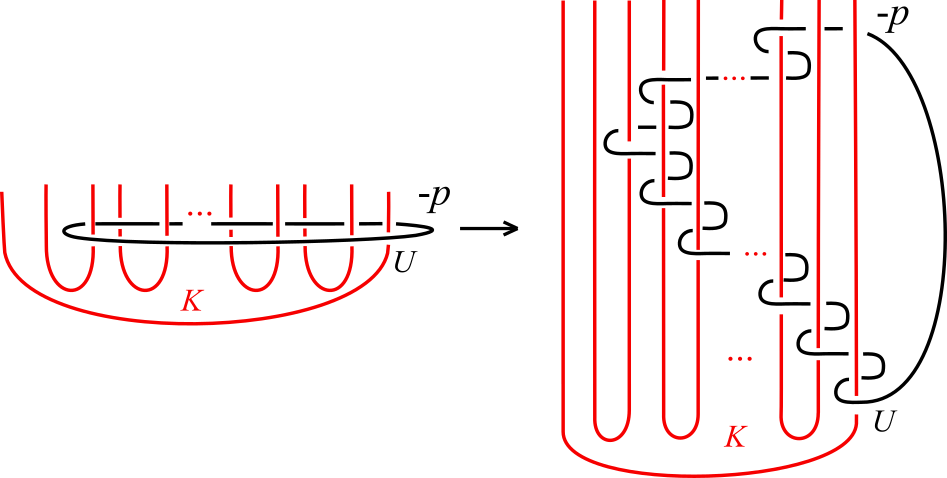}
 \caption{A decomposition of the surgery unknot $U$.}
  \label{knot3}
\end{center}
\end{figure}

\begin{proof}[Proof of Theorem~\ref{thm1}] We start with a decomposition of a knot $K$ and an isotoped unknot $U$ in $S^1\times S^2$ in terms of the standard generators of the pure braid group. The link diagram of $K\cup U$ in this case, contains only full twists. We can remove these full twists by blowing up, by adding an unknot with framing $\pm 1$ which changes the framings by adding $\pm 1$. We remove the full twists of $U$ as in Figure~\ref{casesp1} so that $U$ has linking number $1$ with $K$. Next  we continue blowing up at the points indicated in Figure~\ref{casesp1} and the surgery framing of $U$ is $0$ again, see Figure~\ref{casesp2}.

\begin{figure}[h!]
\begin{center}
 \includegraphics[width=12.8cm]{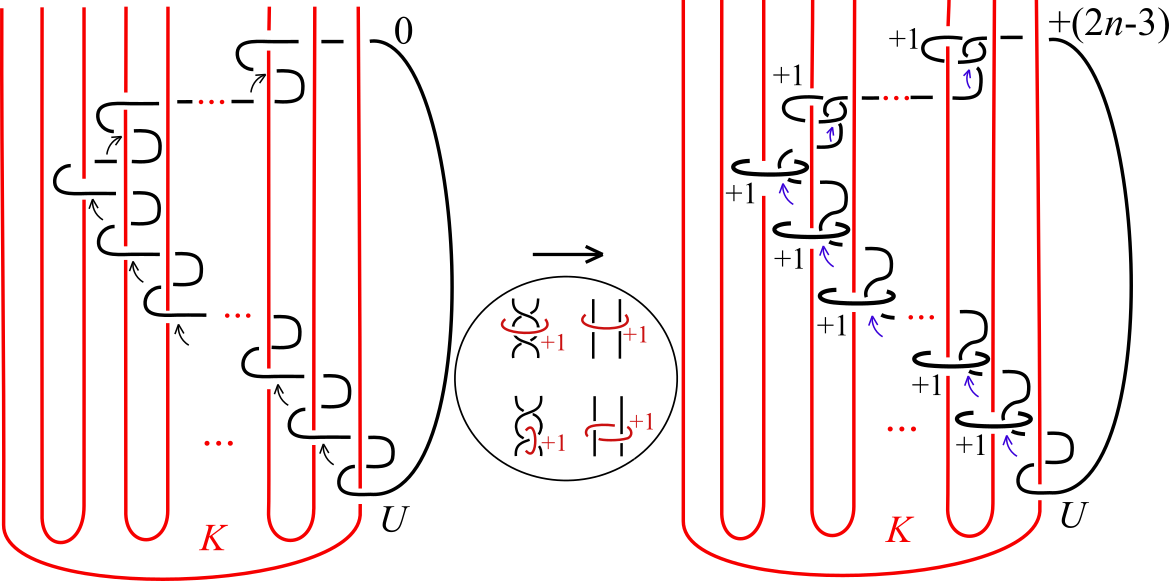}
 \caption{Removing the full twist by blowing up.}
  \label{casesp1}
\end{center}
\end{figure}

\begin{figure}[h!]
\begin{center}
 \includegraphics[width=5.5cm]{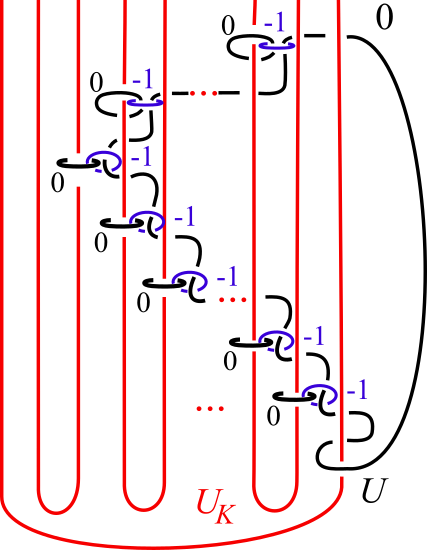}
 \caption{The unknot $U$ has linking number $1$ with $K$ and it has framing $0$.}
  \label{casesp2}
\end{center}
\end{figure}

Then, we unknot $K$ by following the algorithm from Remark $5.1$ and Theorem $5.2$ of \cite{On} that puts the knots in $S^3$ on a planar page of an open book decomposition whose monodromy is a product of positive Dehn twists. The unknotted $K$ (call $U_K$) bounds a disk in $S^3$ and  $S^3$ has a natural open book decomposition with this disk page and with the monodromy $\phi = Id$. We will use this open book decomposition to construct the desired open book of $S^1 \times S^2$ containing $K$ on its page. A page of an open book decomposition of $S^1 \times S^2$ is obtained from the disk $U_K$ bounds after $0$-framed surgeries. More specifically, each $0$-framed unknot (in particular the unknot $U$) is arranged to puncture each disk page transversely once to form a binding component after surgery. Each $-1$-framed unknot is arranged to be isotoped onto one of the punctured disk pages to perform the surgery on the page. Performing $-1$-surgeries on these unknots on the pages will yield an open book decomposition of $S^1 \times S^2$ whose monodromy is the old monodromy $\phi = Id$ composed with a right-handed Dehn twist along each unknot. Moreover, by \cite{Gi2} this open book supports a Stein fillable and hence a tight contact structure on $S^1 \times S^2$. After the surgeries, the unknot $U_K$ on the page will be isotopic to the knot $K$ and it can be Legendrian realized. The proof extends for any given link type.
\end{proof}

\begin{remark} Note that if we follow the algorithm from the proof of Theorem~\ref{thm1} for $\l$, then the surgery framing of $U$ in Figure~\ref{casesp2} will be $-p$. To arrange the surgery framing of $U$ to be $0$, we need to blow up $p$ times, by adding $p$ unknots with framing $+1$. In this case, $+1$-framed unknots contribute negative Dehn twists to the monodromy of the open book yielding an overtwisted contact structure on $\l$ (it is easy to find a sobering arc in this case, \cite{Gmn}). This is the reason why we cannot generalize Theorem~\ref{thm1} to $\l$.  See Figure~\ref{remark}. \label{remark2}
\end{remark}
\begin{figure}[h!]
\begin{center}
 \includegraphics[width=8cm]{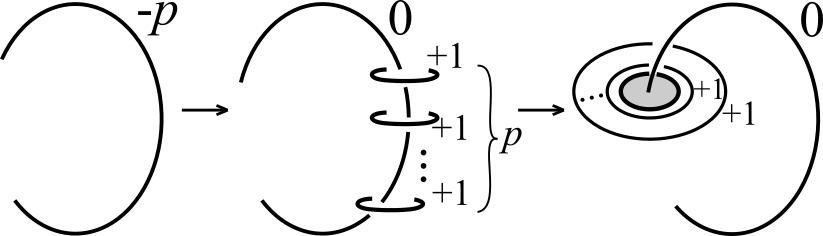}
 \caption{ Unknots with $+1$ framing contribute negative Dehn twists to the monodromy.}
  \label{remark}
\end{center}
\end{figure}

\begin{proof}[Proof of Theorem~\ref{thm2}] Let $K$ be a knot in $\l$ decomposed in terms of the standard generators of the pure braid group and let $U$ be the unknot with surgery framing $-p$ decomposed as in Figure~\ref{knot3}. To construct an open book decomposition for $\l$, we arrange the surgery framing of $U$ to be $-1$. To this end we blow up at indicated points in Figure~\ref{caseslp1}, by adding $+1$-framed unknots at $2n-2$ points. This adds $+(2n-2)$ to the surgery framing $-p$.
 
\begin{figure}[h!]
\begin{center}
 \includegraphics[width=12cm]{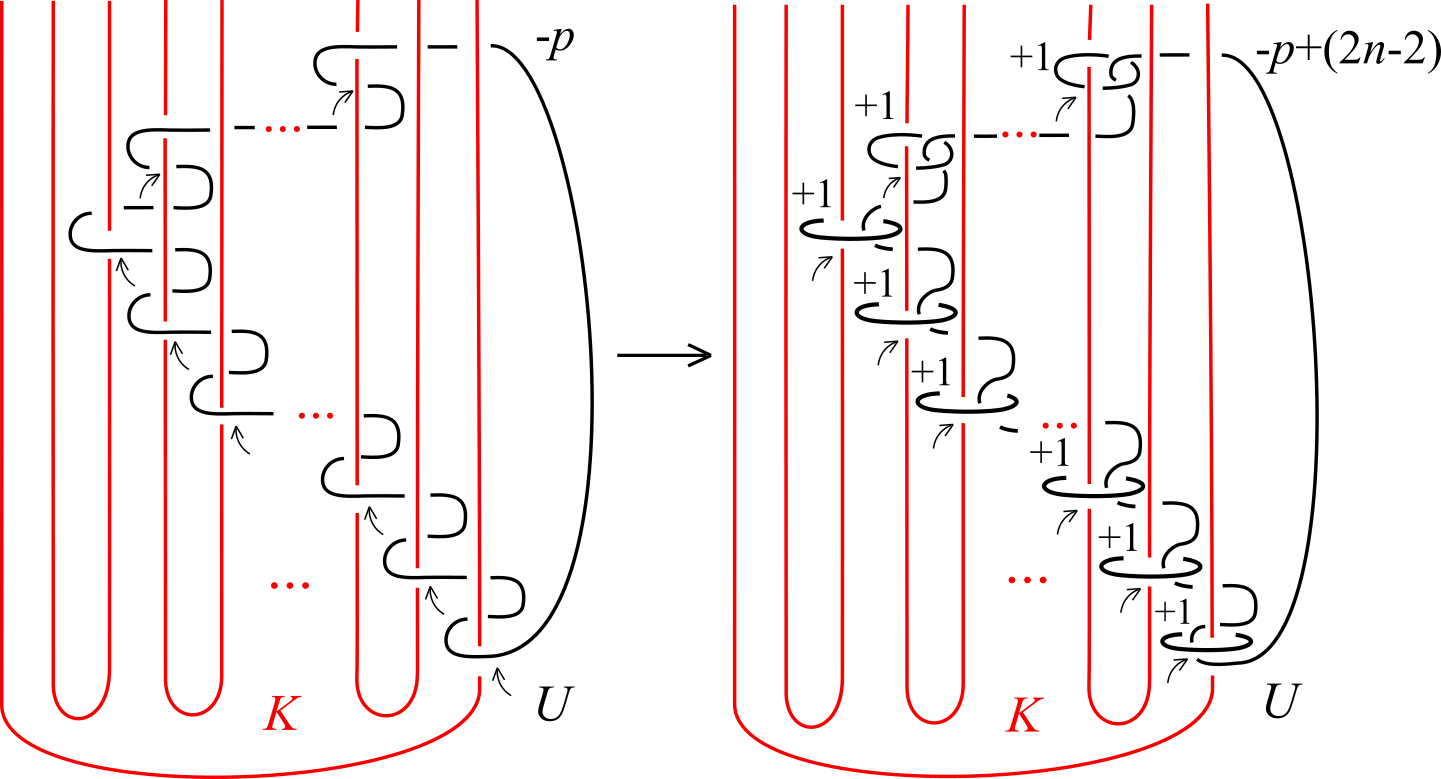}
 \caption{Arranging the framing of $U$.}
  \label{caseslp1}
\end{center}
\end{figure}

Next, we continue blowing up as in Figure~\ref{caseslp2} to arrange the framing of any unknot having linking number $1$ with $K$ to be $0$.

\begin{figure}[h!]
\begin{center}
 \includegraphics[width=12cm]{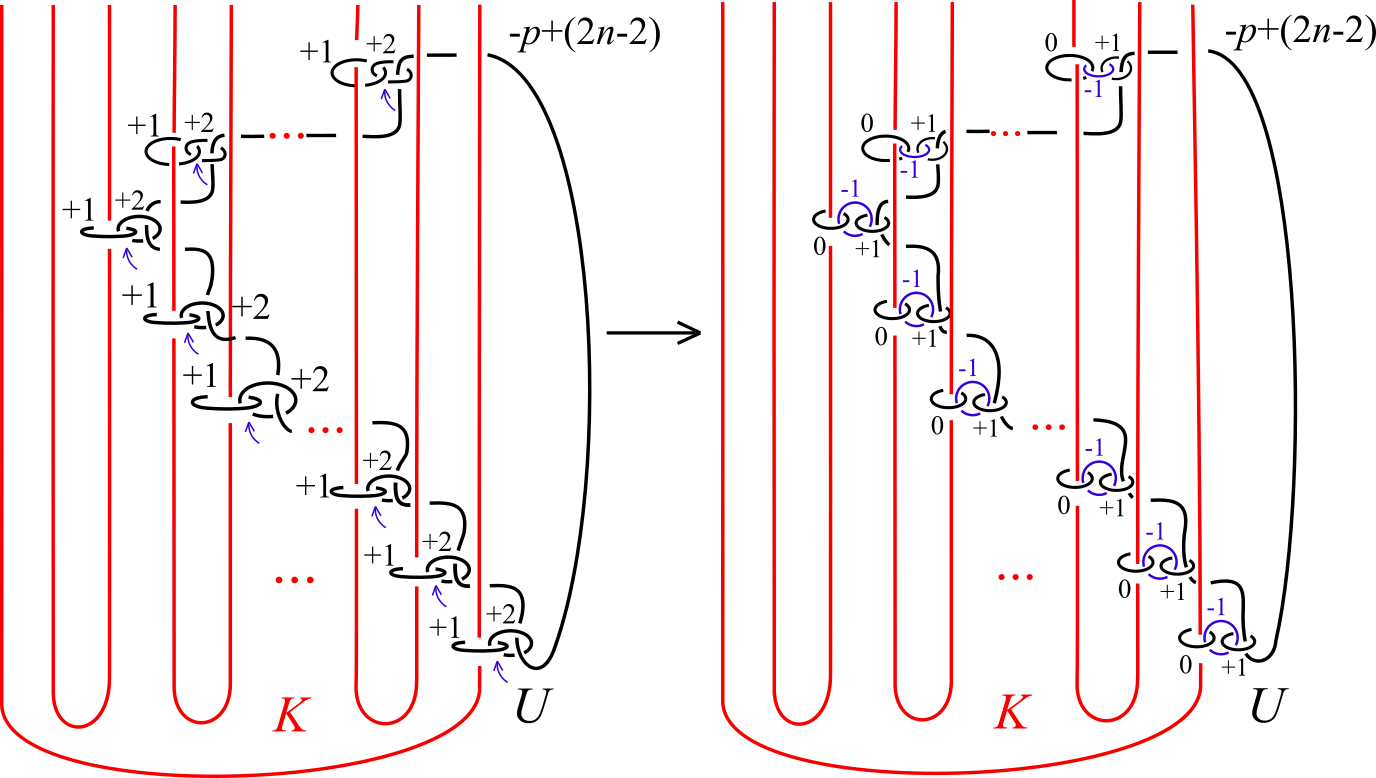}
 \caption{Each unknot linking $K$ has framing $0$.}
  \label{caseslp2}
\end{center}
\end{figure}

Now, to arrange the surgery framing of $U$ to be $-1$, we blow up by adding $p-2n+1$ unknots with framing $+1$ as in Figure~\ref{caseslp3}. This is possible only when $-p+(2n-2) < 0 $. In addition, we arrange the framing of any unknot having linking number $1$ with $U$ to be $0$ in Figure~\ref{caseslp3}  (if $-p+(2n-2) \geq 0$, then to arrange the framing of $U$ to be $-1$, we need to add unknots with framing $-1$ and to continue to arrange these to have framing $0$, we need to add unknots with framing $+1$. However, similar to the case of Remark~\ref{remark2}, this yields an open book decomposition that supports an overtwisted contact structure). 

Finally, we unknot the knot $K$ by using the algorithm from Remark $5.1$ of \cite{On}. We denote the unknotted $K$ by $U_K$. Thus, we get a   collection of link of unknots having framing $0$ or framing $-1$ only, a Kirby diagram corresponding to the open book decomposition. A page of an open book decomposition is constructed by taking the connected sum of the disk bounded by $U_K$ and the disk bounded by $U$. We arrange each $0$-framed unknot to puncture each disk page transversely once to form a binding component after the surgery. Each $-1$-framed unknot (in particular the unknot $U$) is arranged to be isotoped onto one of the punctured disk pages to perform the surgery on the page. Since each $-1$-framed unknot contributes positive Dehn twists to the monodromy, we get an open book decomposition for $L(p,1)$ supporting a tight contact structure. After the surgeries, $U_K$ will be isotopic to $K$ and we can  Legendrian realize $K$ on the page.
\begin{figure}[h!]
\begin{center}
 \includegraphics[width=14.2cm]{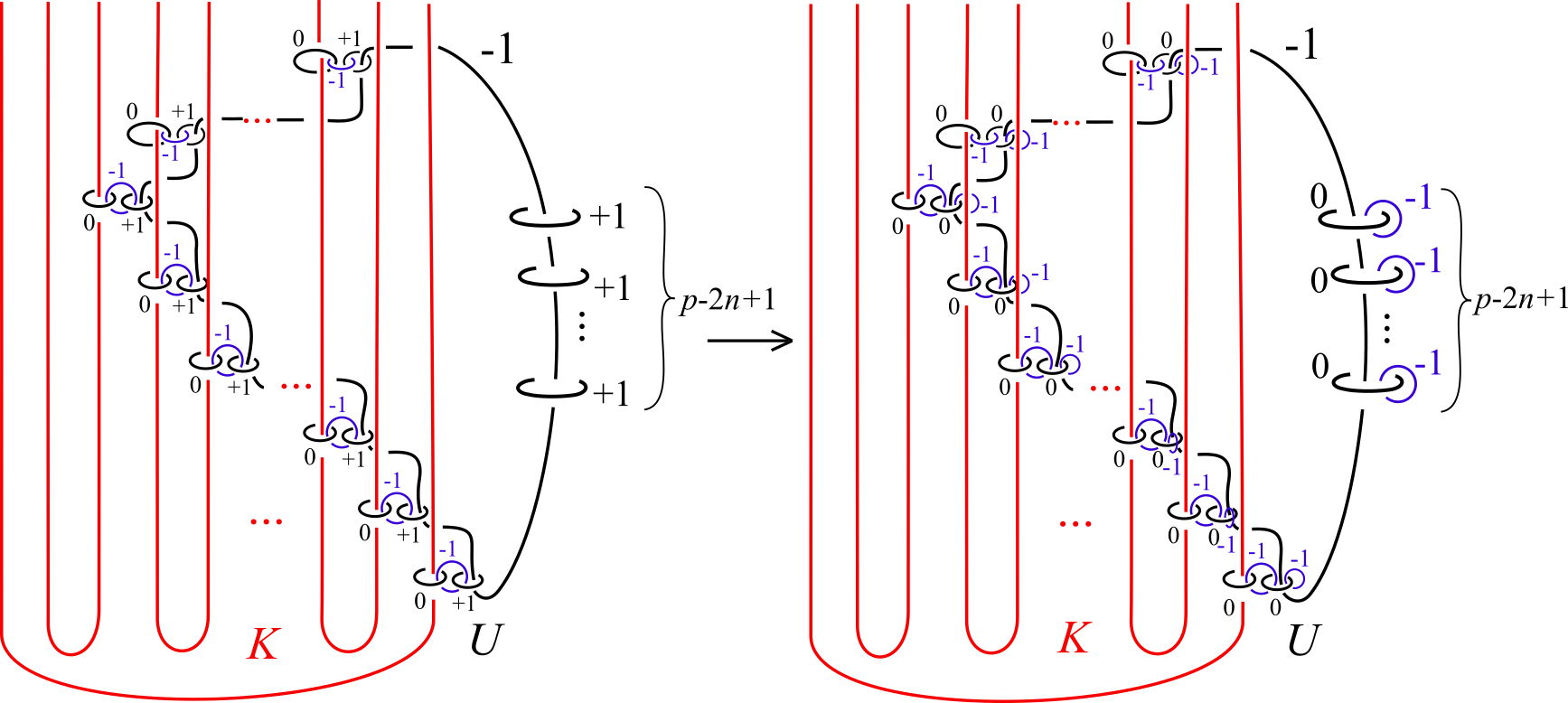}
 \caption{ Each unknot linking $U$ has framing $0$ and  $U$ has surgery framing $-1$.}
  \label{caseslp3}
\end{center}
\end{figure}
\end{proof}

We now give illustrative examples.

\begin{example} Let $K$ be a knot in $L(4,1)$ as in Figure~\ref{nel1} where the knot $K$ and the unknot $U$ are decomposed in terms of the standard generators of the pure braid group.
\begin{figure}[h!]
\begin{center}
 \includegraphics[width=6.5cm]{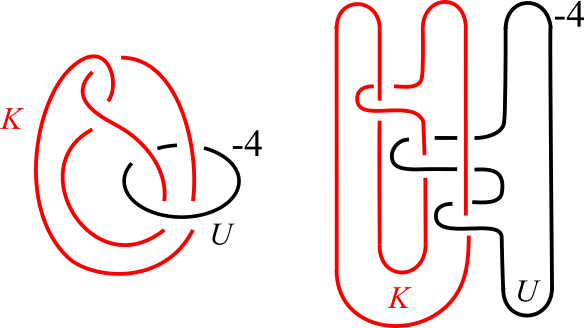}
 \caption{A knot $K$ in $L(4,1)$. $K$ is presented as a $4$-pure braided plat and $K \cup U$ is decomposed in terms of the standard generators of the pure braid group.}
  \label{nel1}
\end{center}
\end{figure}

Follow the algorithm from Theorem $5.2$ of \cite{On} for $K$ and follow the algorithm from Theorem~\ref{thm2} for $U$ and obtain Figure~\ref{nel2} where the unknot $U$ has the framing $-1$ at the end. 

\begin{figure}[h!]
\begin{center}
 \includegraphics[width=14cm]{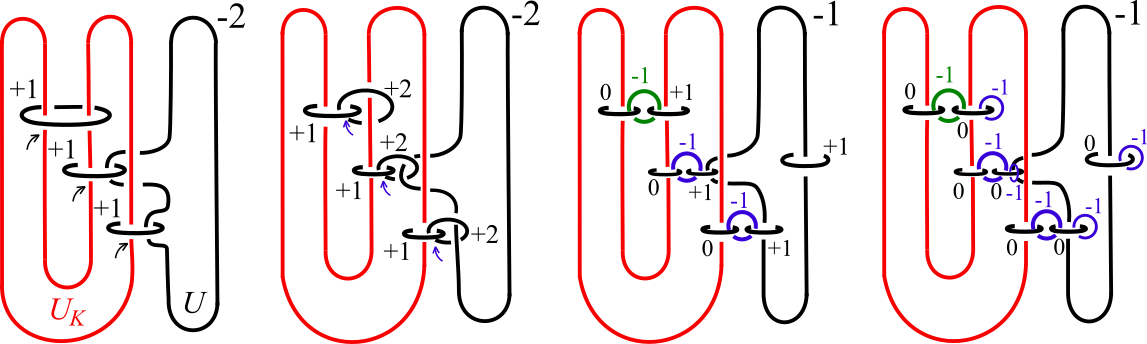}
 \caption{Unknotting $K \cup U$.}
  \label{nel2}
\end{center}
\end{figure}

A page of an open book decomposition is constructed by taking the connected sum of the disk bounded by $U_K$ and the disk bounded by $U$. We can now isotope $-1$-framed unknots onto a page using the band for the connected sum. We obtain a planar open book decomposition for $L(4,1)$ where $U_K$ and each $-1$-framed unknot sit (also $U$) on the page and each $0$-framed component forms a binding component as in Figure~\ref{oblens}. The page of the open decomposition is a disk with seven punctures. Note that after performing the surgeries, each $-1$-framed unknot contributes a right-handed Dehn twist to the monodromy of the open book decomposition and the unknot $U_K$ will be isotopic to the knot $K$. By \cite{Gi2}, this open book supports a Stein fillable and hence a tight contact structure on $L(4,1)$. We can Legendrian realize each $-1$-framed unknot and $K$ on the page of this open book. 

\begin{figure}[h!]
\begin{center}
 \includegraphics[width=11cm]{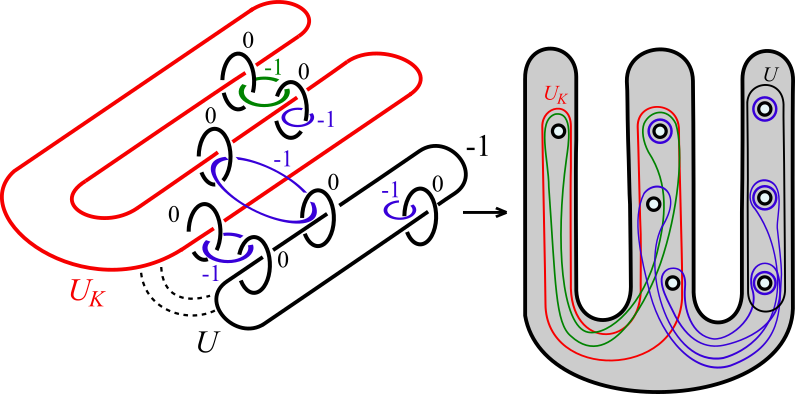}
 \caption{Taking connected sum of the disks bounded by $U_K$ and $U$ to construct a page.}
  \label{oblens}
\end{center}
\end{figure}
\end{example}

\begin{example} Let $K$ be a knot in $S^1 \times S^2$ where $K$ and $U$ are decomposed in terms of the standard generators of the pure braided group as in Figure~\ref{nex1}. 
\begin{figure}[h!]
\begin{center}
 \includegraphics[width=6.5cm]{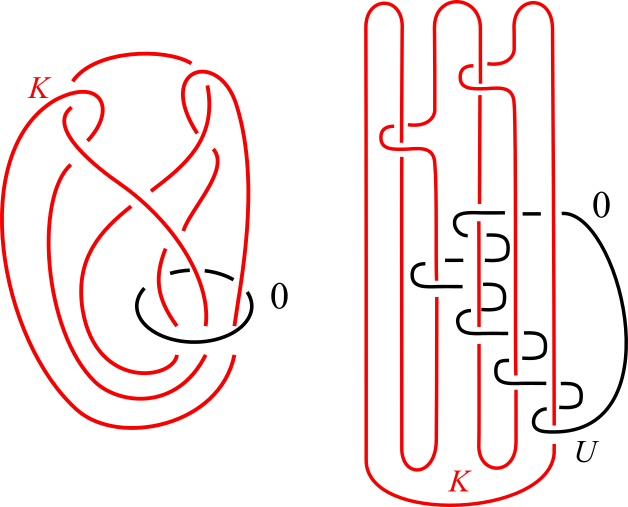}
 \caption{A knot $K$ in $S^1 \times S^2$. $K$ is presented as a $6$-pure braided plat and $K \cup U$ is decomposed in terms of the standard generators of the pure braid group.}
  \label{nex1}
\end{center}
\end{figure}

By following the proof of Theorem~\ref{thm1}, unknot the knot $K$ and call it $U_K$. Continue blowing up at the points indicated in Figure~\ref{nex2} so that $U$ has linking number $1$ with $U_K$ and has framing $0$ at the end. Further arrange all $0$-framed unknots to have linking number $1$ with $U_K$ as in  Figure~\ref{obex1} (and hence each $0$-framed unknot punctures the disk $U_K$ bounds).

\begin{figure}[h!]
\begin{center}
 \includegraphics[width=7.7cm]{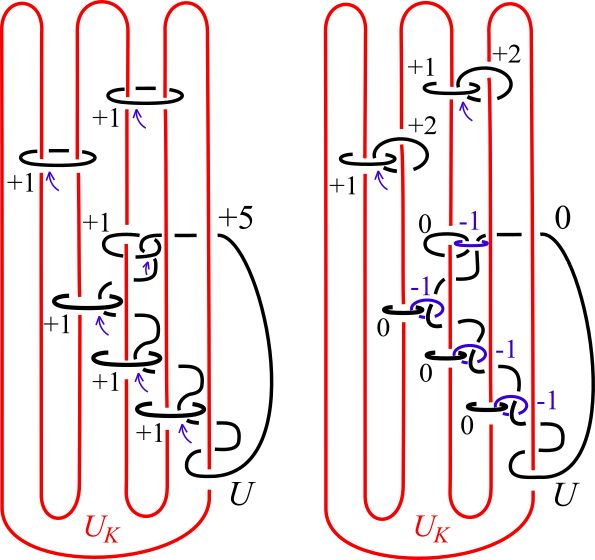}
 \caption{Unknotting $K \cup U$.}
  \label{nex2}
\end{center}
\end{figure}

A page of an open book decomposition for  $S^1 \times S^2$ containing $K$ is constructed from the disk bounded by $U_K$ in $S^3$, see Figure~\ref{obex1}. The $3-$sphere $S^3$ has a natural open book decomposition with this disk page and with the monodromy $\phi = Id$.  After performing the $0$-surgeries, each unknot with framing $0$ (also $U$) forms a binding component. The page of the open book decomposition is a disk with nine punctures. We isotope all $-1$-framed unknots onto a page. After performing $-1$-surgeries, we will get an open book decomposition of $S^1 \times S^2$ whose monodromy is the old monodromy $\phi = Id$ composed with a right-handed Dehn twist along each $-1$-framed unknot. Pushing $U_K$ off of itself inside the page is a way to see $U_K$ on the page.  After the surgeries, $U_K$ will be isotopic to $K$ and $K$ can be Legendrian realized on the page.
\begin{figure}[h!]
\begin{center}
 \includegraphics[width=8cm]{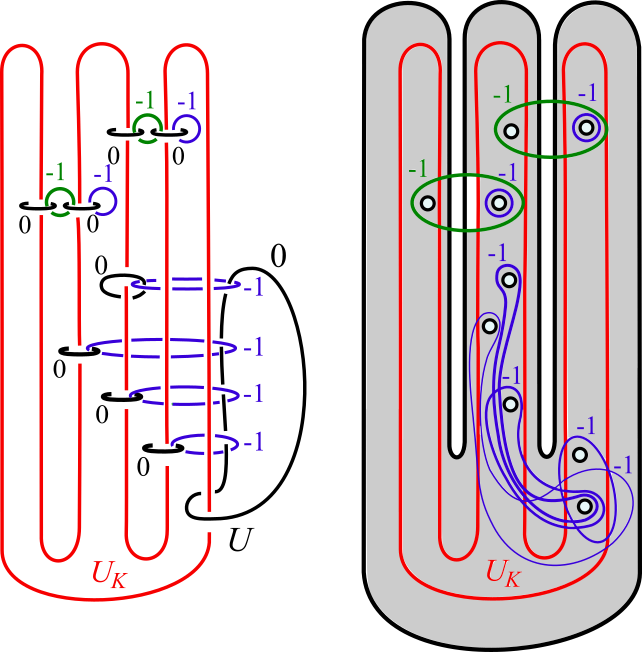}
 \caption{Passing from a surgery diagram to an open book decomposition. (Altough two green $-1$-framed unknots are not drawn on the page, both surgeries can be performed on the page.)}
  \label{obex1}
\end{center}
\end{figure}
\end{example}

\section{Final Remarks and questions}

\begin{remark} By Lemma 5.4 of \cite{On}, any Legendrian knot in any tight lens space $L(p,1)$ or in the tight $S^1 \times S^2$ with a Thurston–Bennequin invariant $\tb(L) > 0$  has $sg(L) > 0$. Thus, the Legendrian representative constructed in the Theorem~\ref{thm1} or in Theorem~\ref{thm2} is a Legendrian knot having the Thurston-Bennequin invariant less than or equal to zero. Thus, it would be very interesting to know answers to the following questions.
\end{remark}
\begin{question} Let $L$ be a Legendrian knot in  tight $(\l, \xi)$ with $\tb(L) \leq 0$. For all such Legendrian knots, is $sg(L) = 0$? 
\end{question}
\begin{question} Is there a contact $3$-manifold where all Legendrian knots have $sg(L)=0$?
\end{question}
\begin{question} Do all contact $3$-manifolds contain a Legendrian knot $L$ with $sg(L)>0$? 
\end{question}
\begin{question} Does there exist a contact $3$-manifold containing a  Legendrian knot with $sg(L)=2$? Construct explicit examples of Legendrian knots with $sg(L)=2$.
\end{question}
\begin{question} Can a Legendrian knot have arbitrary large support genus?
\end{question}
There are many operations one can perform on knots: connected sum, cabling, Whitehead doubling, Legendrian satellite operation. Along these lines:
 \begin{question} Which operations on knots preserve the support genus? Is there an operation on knots which increases/decreases the support genus?
\end{question}
\begin{question} Which contact surgeries preserves the support genus?
\end{question}
Another interesting question is
\begin{question} For what knot types in a contact $3$-manifold are all Legendrian knots with maximal Thurston-Bennequin invariant determined by their support genus invariant? 
\end{question}

\begin{remark} The planar open book constructed in Theorem~\ref{thm2} supports some tight contact structure on $\l$ which we can not keep track of from the algorithm. It would be very interesting to find an algorithm that works for all possible tight contact structures on lens spaces and which puts the given knot on the page and keeps the open book decomposition planar. Note that since $S^1 \times S^2$ has a unique tight (and Stein fillable) contact structure such a problem does not appear for $S^1 \times S^2$. 
\end{remark}
\begin{question} Do all knot types in any tight $\l$ have a Legendrian representative $L$ such that $sg(L) = 0$?
\end{question}
It is not (a priori) clear that Theorem~\ref{thm2} extends to general lens spaces.
\begin{question} For $p> q >0$, do all knot types in any tight $L(p,q)$ have a Legendrian representative $L$ such that $sg(L) = 0$?
\end{question}

\end{document}